\documentclass{amsart}
\usepackage{amsthm,amsfonts,amssymb,amsmath,amscd
, latexsym, wasysym
}
\usepackage[centerlast]{subfigure}
\usepackage[rotate,arrow,matrix]{xy}

\DeclareMathOperator{\sSets}{\mathbf{sSets}}
\DeclareMathOperator{\gSp}{\mathbf{G}\textit{Sp}}

\DeclareMathOperator{\gSpOm}{(\mathbf{G}\textit{Sp}, \Omega)}
\DeclareMathOperator{\gOne}{\mathbf{G}\text{[1]}}
\DeclareMathOperator{\g0One}{\mathbf{G}\text{[0,1]}}
\DeclareMathOperator{\EG}{\mathbf{EG}}
\DeclareMathOperator{\BG}{\mathbf{BG}}
\DeclareMathOperator{\bigG}{\mathbf{G}}

\DeclareMathOperator{\spbg}{\sSets \downarrow \BG}
\DeclareMathOperator{\spbgsig}{(\sSets \downarrow \BG, \Sigma)}

\DeclareMathOperator{\hspbgsig}{\mathbb{T}\spbgsig}
\DeclareMathOperator{\hgSpOm}{\mathbb{T}\gSpOm}
\DeclareMathOperator{\Xg}{(X \times_{\bigG} \EG)}
\DeclareMathOperator{\Yg}{(Y \times_{\bigG} \EG)}

\DeclareMathOperator{\rC}{\mathbf{RelCat}}
\DeclareMathOperator{\1h}{\hat{1}}
\DeclareMathOperator{\hcolim}{\mathcal{M}}
\DeclareMathOperator{\hpull}{\mathcal{N}}

\newtheorem{thm}{Theorem}[section]

\newtheorem{prop}[thm]{Proposition}

\theoremstyle{definition}
\newtheorem{df}[thm]{Definition}

\theoremstyle{remark}

\newtheorem*{rem}{Remark}
\newtheorem{ack}{Acknowledgments}

\def\C{\mathcal C}

\begin{document}

\title[On the homotopy theory of $\bigG$ - spaces]{On the homotopy theory of $\bigG$ - spaces}
\author[A. Sharma]{Amit Sharma}
\email{sharm121@umn.edu}
\address {School of Mathematics\\University of Minnesota\\
  Minneapolis, MN 55455}

\date{December 8, 2015}

\begin{abstract}
The aim of this paper is to show that the most elementary homotopy
theory of $\bigG$-spaces is equivalent to a
homotopy theory of simplicial sets over $\BG$, where
$\bigG$ is a fixed group. Both homotopy theories are presented
as Relative categories. We establish the equivalence by constructing
a strict homotopy equivalence between the two relative categories. No
Model category structure is assumed on either Relative Category.

\end{abstract}

\maketitle


\section{Introduction}
A $\bigG$-space is a simplicial set with the action of a group $\bigG$.
The category of $\bigG$-spaces and $\bigG$-equivariant simplicial maps
is denoted by $\gSp$.
We establish an equivalence between the homotopy theory of
$\bigG$-spaces, $\hgSpOm$ and the homotopy theory $\hspbgsig$.
Both homotopy theories are presented as relative categories.
The notation above is meant to distinguish
the homotopy theory from the relative category presenting it.
We denote the category of simplicial sets by $\sSets$ and we will assume
the Kan model category structure on $\sSets$, see \cite{DQ} and \cite{MH99}. The
classifying space of the group $\bigG$ is denoted $\BG$.
The subcategory of weak equivalences, $\Omega$, has the same
objects as $\gSp$ and its morphisms are weak equivalences in $\sSets$
which preserve the group action. The subcategory of weak equivalences, $\Sigma$,
has the same objects as $\spbg$. A morphism in $\Sigma$
is an arrow of the comma category such that the arrow over
the vertex $\BG$ is a weak equivalence in $\sSets$ i.e.
its image under the geometric realization functor is a
weak equivalence of topological spaces.

The goal of this paper is to provide a direct proof, in the context
of relative categories, of a result on homotopy equivalence
of simplicial categories obtained by \emph{simplicial localizations},
see \cite[3.3]{DK81}, \cite{DK80} and \cite{DK82},
of the two relative categories $\gSpOm$ and $\spbgsig$,
which was announced in a note \cite[Theorem 2.1]{DKD80}.
The proof indicated in \cite{DKD80} is based on
constructing suitable simplicial model category structures on
the two relative categories $\gSpOm$ and $\spbgsig$ and then establishing
a homotopy equivalence, in the sense of \cite[2.5]{DK81}, between
the underlying \emph{simplicial homotopy categories}, namely, the
simplicial subcategories of cofibrant and fibrant objects of the two
simplicial model categories in question.


In this paper, we do not assume any model category structure. We have
taken a very direct approach of writing a strict homotopy equivalence
between the relative categories $\gSpOm$ and $\spbgsig$ which establishes
a homotopy equivalence of homotopy theories $\hgSpOm$ and $\hspbgsig$.
Finally, proposition \ref{HERel-HES} along with \cite[2.3]{KB11} and
remark \ref{simp-groups} show that our main result is an equivalent version
of \cite[Theorem 2.1]{DKD80}, thereby proving that theorem.

\cite[Theorem 2.1]{DKD80} is a particular case of \cite[Theorem 2.2.1.2]{JL}.
Jacob Lurie uses some sophisticated simplicial techniques
to prove this generalized version of the theorem in \cite{DKD80},
as a homotopy equivalence of simplicial categories. In a future paper
we plan on extending the arument of the proof of our main result to
obtain an equivalence of homotopy theories of simplicial maps over an
arbitrary simplicial set $B$, and a suitably defined homotopy theory
of functors into the category of simplicial sets.

\begin{ack}
  The author is thankful to Alexander Voronov for frequent discussions
  regarding this paper and also for his suggestions on the proof of
  Theorem \ref{main-result}. The author is also thankful to W. G. Dwyer
  who proposed the idea of Proposition \ref{HERel-HES} to the author
  in a private email message.
\end{ack}

\section{Setup}
A homotopy theory is presented most naturally as a relative category.
In this section we give a brief introduction to relative categories and
explain the notion of weak equivalences in relative categories.
A relative category is a pair $(\C,\Gamma)$, where $\C$ is an
ordinary (small) category and $\Gamma$ is a subcategory of $\C$
having the same set of objects as $\C$. The morphisms of $\Gamma$
will be called \emph{weak equivalences} of $\C$. A \emph{functor of
relative categories} is an ordinary functor which preserves weak equivalences.
The category of all small relative categories and functors of relative categories
,called $\rC$, has been given a model category structure by Barwick and Kan see \cite{BK11}.
Let $\1h$ denote the category $0 \to 1$ having two objects and
exactly one, nonidentity morphism. This category is treated as a relative category
in which every morphisms is a weak equivalences, namely $(\1h, \1h)$.
\begin{df}
\label{strict-homotopy}
A pair of functors of relative categories $F, G:(\C, \Gamma) \to (\C',\Gamma')$
are \emph{strictly homotopic} if there exists a functor of relative categories
\[
H:\C \times \1h \rightarrow \C,
\]
such that $H(x,0) = F(x)$ and $H(x,1) = G(x)$, for all $x \in Ob(\C)$.
$H$ is called a \emph{strict homotopy} between $F$ and $G$.
\end{df}
Notice that $H$ is a natural weak eqivalence between $F$ and $G$,
$H$ assigns to each $x \in Ob(\C)$, a weak equivalence in $\C'$.
A morphism $f:x \rightarrow y$ in $Mor(\C)$ is assigned a
commutative diagram

\[
 \xy
{\ar^{H(x,0 \rightarrow 1)} (0,20)*+{H(x,0)}; (30,20)*+{H(x,1)}};
{\ar_{H(f,0)}(0,20)*+{H(x,0)}; (0,0)*+{H(y,0)}};
{\ar^{H(f, 1)} (30,20)*+{H(x,1)}; (30,0)*+{H(y,1)}};
{\ar_{H(y,0 \rightarrow 1)} (0,0)*+{H(y,0)}; (30,0)*+{H(y,1)}};
\endxy
\]

Moreover, if \textit{f} is a weak equivalence in $\C$, then each arrow in
the above commutative diagram is a weak equivaleces in $\C'$.
\begin{df}
 A morphism of relative categories, $f:(\C, \Gamma) \to (\C', \Gamma')$,
is a \emph{strict homotopy equivalence} if there exists another morphism of
relative categories, $f':(\C', \Gamma') \to (\C, \Gamma)$ (called
the inverse of $f$) such that the compositions $f'f$ and $ff'$ are
strictly homotopic, see \ref{strict-homotopy}, to the identity maps
of $\C$ and $\C'$ respectively.
\end{df}

\subsection{The Homotopy category}
The homotopy category of a relative category $(\C, \Gamma)$ is obtained by
"formally inverting`" all morphisms in the subcategory $\Gamma$. Given two
objects $X, Y \in Ob(\C)$ and an integer $n \ge 0$, a \emph{zigzag} in $\C$ from
$X$ to $Y$ of length $n$ is a sequence
\[
 \xymatrix{
 X = \C_{0} \ar[r] &\cdot \ar@{-}[r]  &\cdot \ar@{-}[r] & \cdot \ar@{-}[r] \cdots &\ar@{-}[r] & \C_{n} = Y \\
 }
\]
of maps in $\C$, each of which is either forward (i.e points to the right) or backward
(i.e points to the left) and such a zigzag is called restricted if all the backward
maps are weak equivalences or arrows in $\Gamma$. The homotopy category will then be the
category $Ho\C$ which has the same objects as $\C$, in which, for every two objects
$X, Y \in Ob(\C)$, the hom-set $Ho\C(X, Y)$ is the set of equivalence classes of the restricted
zigzags in $\C$ from $X$ to $Y$, where two such zigzags are in the same class if
one can be transformed into the other by a finite sequence of operations of the
following three types and their inverses:\\
$(i)$ omiting the identity map\\
$(ii)$ replacing two adjacent maps which go in the same direction by their
composition, and\\
$(iii)$ omitting two adjacent maps when they are the same, but go in the
opposite direction\\
and in which the compositions are induced by the composition of the
zigzags involved. If the category $\C$ is small then the category
$Ho\C$ is also small.
%

\subsection{A functorial construction of classifying spaces of groups}
\begin{sloppypar}
Let $\bigG$ be a discrete group. Let $\gOne$ be the category with one object
$\star$ and $Hom_{\bigG}(\star, \star)$ is isomorphic to $\bigG$. We claim
that $N(\gOne)$ is the classifying space $\BG$. Let $\g0One$ be
the category with one object for each element of $\bigG$
and exactly one arrow between any two objects. Each object is both initial and
final hence the nerve of $\g0One$, denoted $\EG := N(\g0One)$ is a
contractible simplicial set. Now we define the notion of an \emph{action} of
a group on a category An action of a group on a category is an assignment of
an automorphism of the category to each element of the group. The action
satifies the usual assiciativity and unit conditions. We define the
(right) action of $\bigG$ on $\g0One$ by assigning to each element
$g \in \bigG$, a functor $\phi_{g}:\g0One \rightarrow \g0One$
which is defined on objects by group multiplication as follows: $\phi_{g}(x)= xg$
and since there is only one morphism between any two objects in $\g0One$,
there is only one way of defining the functor on arrows. Also,
$\phi_{g_{1}g_{2}} = \phi_{g_{1}}\circ \phi_{g_{2}}$ and if $e$ is the
unit element of the group then $\phi_{ge} = \phi_{g}$. This action of
$\bigG$ on $\g0One$ induces an action of $\bigG$ on $\EG$
This induced action is free and it is easy to check that $\EG/\bigG$
is isomorphic to $\BG$.

A group homomorphism $\bigG \rightarrow \mathbf{H}$ induces a functor
$\gOne \rightarrow \mathbf{H}[1]$ and hence induces a map on
their nerves. Let $Aut(\g0One)$ denote the monoid in $Cat$ whose
object is the category $\g0One$ and morphisms are automorphisms of $\g0One$,
i.e. functors whose source and target is the category $\g0One$ and which have an
inverse. An action as defined above is uniquely determined  by a functor
$\gOne \rightarrow Aut(\g0One))$.
\end{sloppypar}

\section{The equivalence of homotopy theories}

The main result of this paper is presented in this section.
We will prove the equivalence of the two homotopy theories
by constructing a strict homotopy equivalence between
the two relative categories presenting the homotopy theories.
Given a simplicial set $S$, The product space $S \times \EG$
has a $\bigG$ - action and $S \times_{\bigG} \EG$ is the
quotient space of this action which is also the total space of
a fibration of simplicial sets
\[
 q_S:S \times_{\bigG} \EG \to \BG,
\]
which will be called the \emph{quotient map}. Any $G$ - space can be
viewed as a functor from the category $\gOne$ into the category $\sSets$.
The homotopy colimit of a functor $X \in Fun(\gOne, \sSets)$
is the following quotient space:
\[
hocolim X = X \times_{\bigG} \EG,
\]
where the simplicial set $X$ on the right side of the above equation
is the image of the functor $X$.

The following proposition is a key step in the proof of our
main result \ref{main-result}:
\begin{prop}
 The $\bigG$-space $\Xg \times_{q_X} \EG$ is isomorphic
 to the product $\bigG$-space $X \times \EG$.
\end{prop}
\begin{proof}
 We prove this proposition by defining a morphism of $\bigG$-spaces
 \[
 K: X \times \EG \to  \Xg \times_{q_X} \EG.
 \]
 In degree $n$, this morphism is defined as follows:
 \[
  (x_n, e_n) \mapsto ((x_n,e_n)\bigG, e_n).
 \]
 Using the freeness of the action of the group $\bigG$ on $\EG$,
 it is easy to check that this map is an isomorphism of simplicial sets.
 Further it is not hard to see that this map preserves the $\bigG$-action.
\end{proof}

\begin{thm}
\label{main-result}
The relative categories $\gSpOm$ and $\spbgsig$ are srtictly homotopy equivalent.
\end{thm}

\begin{proof}
%
 We begin our proof by defining two functor of relative categories
as follows:
The first
\[
 \hcolim:\gSpOm \rightarrow \spbgsig
\]
is defined on objects as
\[
X \mapsto q_X:X \times_{\bigG} \EG \to \BG ,
\]
where $q_X:X \times_{\bigG} \EG \to \BG$ is the quotient map.
A morphism $m:X \to Y \in \gSp$ induces a morphism on the
quotient spaces
\[
 \overline{m}:X \times_{\bigG} \EG \to Y \times_{\bigG} \EG.
\]
Moreover, $\overline{m}$ is a morphism in $\spbg$.
The functor $\hcolim$ is defined on morphisms as follows:
\[
m:X \rightarrow Y \mapsto \overline{m}:X \times_{\bigG} \EG
\to Y \times_{\bigG} \EG.
\]

 Clearly $\hcolim$ is a functor of relative categories.
 A morphism $f:Y \to \BG \in Mor(\sSets)$ uniquely determines
 a fibration $R(f):R(Y) \to \BG$, where $R(Y)$ is simplicial set
 with an acyclic cofibration $Y \to R(Y)$ in the model category
 $\sSets$. We define the second functor of relative categories
 $\hpull:\spbgsig \rightarrow \gSpOm$ as follows:
\[
 f:Y \to \BG \mapsto R(Y) \times_{R(f)} \EG
\]
 This is a functor of relative categories also. Now, the
 composite functors are defined on objects as follows:

\[
 \hpull(\hcolim(X)) = \Xg \times_{q_X} \EG
\]
and
\[
 \hcolim(\hpull(f:X \to \BG)) = q_{R(X) \times_{R(f)} \EG}:(R(X)
 \times_{R(f)} \EG) \times_{\bigG} \EG \to \BG.
\]

We first define the homotopy 

\[
\hpull\hcolim: \gSp \times \1h \rightarrow \gSp
\]
On objects, it is defined as
\[
h(X,0) = \hpull(\hcolim(X)) = \Xg \times_{q_X} \EG,  h(X,1) = X
\]
and on morphisms it is defined as
\[
h(m:X \rightarrow Y, id_{0}) = \Xg \times_{q_X} \EG
\overset{\overline{m} \times_{\BG} \EG} \rightarrow \Yg \times_{q_Y} \EG
\]
\[
h(X \rightarrow Y, id_{1}) = X \rightarrow Y
\]
\[
h(X \rightarrow Y, \1h) = \Xg \times_{q_X} \EG \rightarrow Y.
\]

$\Xg$ is a principal $\bigG$-space over $\BG$. Let $q_X:\Xg \rightarrow \BG$
be the quotient map. 
The $\bigG$ - map $h(X,1)$ is the following composition:
\[
 \Xg \times_{q_X} \EG \cong X \times \EG \rightarrow X \rightarrow Y.
 \]
Clearly, if $X \rightarrow Y$ is a weak equivalence in $\gSpOm$,
then the above map is also a weak equivalence in $\gSpOm$.
Hence $h$ is a weak equivalence preserving functor or $h \in Mor(\rC)$.

Now we define the reverse homotopy $k:1_{\spbgsig} \Rightarrow \hcolim\hpull$.
\[
k:\spbg \times \1h \rightarrow \spbg.
\]
On objects, $k$ is defined as
\[
k(f:X \to \BG,0)=f, k(f,1)= q_{R(X)}:R(X) \times_{R(f)} \EG \to \BG.
\]

On morphisms, $k$ is defined as follows:
\[
k(f \to g,id_{0}) = f \to g,
\]
\[
k(f \to g, id_{1}) = R(X) \times_{R(f)} \EG \to R(Z) \times_{R(g)} \EG,
\]
\begin{sloppypar}
where $f:X \to \BG$ and $g:Z \to \BG$ are objects of $\spbg$.
For every $f \in Ob(\spbg)$, there is a \emph{zero morphism}
obtained by the following composition: $0:R(X) \to \ast \to \EG
\in Mor(\spbg)$. Since the action of $\bigG$ on $R(X) \times_{R(f)} \EG$
is free, there is a weak equivalence of simplicial sets with
contractible fibers
\end{sloppypar}
\[
 F:(R(X) \times_{R(f)} \EG) \times_{\bigG} \EG \to R(X).
\]
The morphism 
\[
R(Z) \times_{R(g)} \EG \overset{(id,0)} \to (R(Z) \times_{R(g)} \EG) \times \EG
\]
induces a morphism on quotient spaces
\[
 s:R(Z) \to (R(Z) \times_{R(g)} \EG) \times_{\bigG} \EG
\]
which is a homotopy inverse of $F$. The following morphism determined by $s$
\[
 s_{/\BG}:g \to q_{R(Z) \times_{R(g)} \EG},
\]
is a weak equivalence in $\spbgsig$.

We define the morphism $k(f \to g, \1h)$ by the following composite in $\spbg$:
\[
 k(f \to g, \1h):f \to g \overset{s_{/\BG}} \to q_{R(Z) \times_{R(g)} \EG}
 \]
%
Clearly, if $f \rightarrow g$ is a weak equivalence in $\spbg$,
then the above arrow is also a weak equivalence in $\spbg$.
Thus $k$ is a morphism of relative categories.
\end{proof}
As mentioned in the introduction, the main goal of this
paper is to provide a direct proof of \cite[Theorem 2.1]{DKD80}.
Now we achieve this goal by proving that our main result is
an equivalent version of \cite[Theorem 2.1]{DKD80}. More
precisely, the following proposition along with our main
result \ref{main-result}, \cite{KB11} and remark \ref{simp-groups}
prove \cite[Theorem 2.1]{DKD80}.
\begin{prop}
\label{HERel-HES}
 The simplicial localization functor takes (strict) homotopy equivalences
 of relative categories to homotopy equivalences of simplicial categories
 in the sense of \cite[2.5]{DK81}.
\end{prop}

\subsection{Future direction of this research}
\label{future-dir}
The homotopy theory $\hgSpOm$ can be identified with a
functor homotopy theory $\sSets^{h\bigG}$. The construction
of this functor homotopy theory is elaborate and we will
not describe it in this paper. The functor $\hcolim$ in the
proof above is derived from the homotopy colimit functor
$hocolim: \sSets^{h\bigG} \to \sSets$ and the functor $\hpull$
in the proof above is a version of the homotopy pullback functor
taking values in $\bigG$-spaces
\[
 - \times_{\BG}^h \EG: \hspbgsig \to \sSets^{h\bigG}.
\]
These two functors induce an equivalence between homotopy theories
$\sSets^{h\bigG}$ and $\hspbgsig$. Let $B$ be an arbitrary simplicial set,
replacing $\BG$ by $B$ we get homotopy theory $\mathbb{T}(\sSets \downarrow B, \Lambda)$
whose weak equivalences are weak equivalences of
simplicial sets over $B$. We claim the existence of
another (relative) category $\mathfrak{C}(B)$ such that
the homotopy theory $\mathbb{T}(\sSets \downarrow B, \Lambda)$
is homotopy equivalent to a functor homotopy theory
$\sSets^{h\mathfrak{C}(B)}$.

\begin{rem}
\label{simp-groups}
 The definition of a $\bigG$-space considered in \cite{DKD80} is more generic
 in the sense that the group acting on a simplicial set could be a simplicial
 group. The argument of the proof our main result remains valid for this case
 also if we replace $\EG$ and $\BG$ by their equivalent versions for a simplicial
 group $G$, namely, $WG$ and $\overline{WG}$ respectively, see \cite{JPM}.
 We want to achieve our future goal mentioned in \ref{future-dir} by considering
 the action of ordinary monoids on simplicial sets. Therefore writing our main
 result for the action of an ordinary group was more pertinent.
\end{rem}

\bibliographystyle{amsalpha}
\bibliography{BundleClassification}

\end{document}